\documentclass [11pt]{amsart}
\usepackage{amscd,amsmath,amsthm,amssymb,amsfonts,epsfig}
\usepackage{bbm}
\usepackage{graphicx}
\usepackage[utf8]{inputenc}
\usepackage{xcolor}
\usepackage{mathrsfs}  
\usepackage{esvect}
\usepackage{enumitem}
\usepackage{indentfirst}
\usepackage{caption}
\usepackage{subcaption}
\usepackage[export]{adjustbox}
\usepackage{tkz-euclide}
\usepackage{tcolorbox}
\usepackage{pgfplots}
\pgfplotsset{compat=newest}
\usepgfplotslibrary{fillbetween}
\usepackage{comment}
\usepackage{url}
\usepackage{hyperref}

\oddsidemargin=.in
\evensidemargin=.in
\usepackage[a4paper,top=3cm,bottom=2cm,left=3cm,right=3cm,marginparwidth=1.75cm]{geometry}

\usetikzlibrary{intersections,positioning, calc}

\renewcommand{\bar}{\overline}
\newcommand{\R}{\mathbb{R}}

\newcommand{\D}{\nabla}
\newcommand{\pa}{\partial}

\newcommand{\bn}{\mathbf{n}}

\newcommand{\bp}{\mathbf{p}}

\newcommand{\bq}{\mathbf{q}}

\newcommand{\bt}{\mathbf{t}}

\newcommand{\cM}{\mathcal{M}}

\newcommand{\cH}{\mathcal{H}}

\newcommand{\ent}{\mathrm{Ent}}

\newcommand{\al}{\alpha}

\newcommand{\ga}{\gamma}

\newcommand{\ka}{\kappa}
\newcommand{\vare}{\varepsilon}

\renewcommand{\th}{\theta}

\newtheorem{theorem}{Theorem}[section]
\newtheorem{corollary}[theorem]{Corollary}
\newtheorem{lemma}[theorem]{Lemma}
\newtheorem{proposition}[theorem]{Proposition}

\theoremstyle{remark}

\theoremstyle{definition}

\title{On bounds of entropy and total curvature for ancient curve shortening flows}

\author{Wei-Bo Su}
\address{WS: Mathematics Division, National Center for Theoretical Sciences, Cosmology Building, No. 1, Sec. 4, Roosevelt Rd., Taipei City 106, Taiwan}
\email{weibo.su@ncts.ntu.edu.tw}

\author{Kai-Wei Zhao}
\address{KZ: Department of Mathematics, University of Notre Dame, Notre Dame, 46556, Indiana, USA}
\email{kzhao4@nd.edu }

\begin{document}

\begin{abstract}
Bounds of total curvature and entropy are two common conditions placed on mean curvature flows. 
We show that these two hypotheses are equivalent for the class of ancient complete embedded smooth planar curve shortening flows, which are one-dimensional mean curvature flows.
As an application, we give a short proof of the uniqueness of tangent flow at infinity of an ancient smooth complete non-compact curve shortening flow with finite entropy embedded in $\mathbb{R}^2$.
\end{abstract}

\maketitle 

\section{Introduction}
A family of planar curves $M_t$ is a solution to the \textit{curve shortening flow} (CSF) if there exists a smooth family of immersions $\ga: \mathbb{L}\times I \to \R^2$, where $I\subset \R$ is an interval and $\mathbb{L}  = \mathbb{R}$ or $\mathbb{S}^{1}$, such that $M_t = \ga(\mathbb{L},t)$ and satisfies the evolution equation
\begin{equation}\label{eq: CSF}
    \ga_t = \ga_{ss} = \boldsymbol{\kappa} = \kappa \bn.
\end{equation}
Throughout this paper, $s$ denotes the arc-length parameter of a planar curve $\gamma$, $\bt = \gamma_s$ is the unit tangent vector and $\bn = J\bt$ is the normal vector, where $J : \R^2 \rightarrow \R^2$ is the counterclockwise rotation by an angle of $\pi/2$, and $\kappa = \langle \gamma_{ss}, \bn \rangle$ is the (signed) curvature.
We say that a flow $\mathcal{M}=\cup_{t\in I}(M_t,t)$ is \textit{ancient} if $(-\infty,t_0)\subset I$ for some $t_0$. Simple examples of ancient flows are self-similarly shrinking circles and translating grim-reapers. See \cite{AL86, DHS10, AAAW13, AY21, AIOZ23} for more complicated examples.

We say that a curve shortening flow $\mathcal{M}$ has \emph{finite total curvature} if
\begin{align}
    \sup_{t \in I} \, \int_{M_t} \lvert \kappa \rvert \, ds < + \infty.
\end{align}
For closed curves, it is well known that the total curvature is non-increasing along the curve shortening flow according to the first variation formula \cite[Section 4.3 page 99]{Man11}:
\begin{align}
    \frac{d }{d t} \int_{M_t} \lvert \kappa \rvert \, ds = -2 \sum_{p\in M_t : \kappa(p, t) = 0} \lvert \kappa_s(p, t) \rvert \, ds.
\end{align}
Here, the set $\{p\in M_t : \kappa(p, t) = 0\}$ is finite by Sturm's theory proved by Angenent \cite{Ang:1988:ZSS} (see also Subsection~\ref{sec: 2.1} for summary).

Next, we recall Colding-Minicozzi's \textit{Entropy} \cite{CM12}.
We denote the \emph{Gaussian weighted length functional} $F_{x_0, \lambda}$ at center $x_0 \in \R^2$ and scale $\sqrt{\lambda}$ ($\lambda>0$) for a curve $\Gamma$ by
\begin{align}
    F_{x_0, \lambda}[\Gamma] = \frac{1}{\sqrt{4 \pi \lambda}}\int_\Gamma e^{-\frac{\lvert x - x_0 \rvert^2}{4\lambda}}\ ds.
\end{align}
Then the entropy of $\Gamma$ is defined as 
\begin{align}\label{eq:Entropy}
    \ent[\Gamma] = \sup_{x_0\in \R^2, \lambda> 0 } F_{x_0, \lambda}[\Gamma].
\end{align}
If $\cM = \cup_{t\in I}(M_t,t)$ is a curve shortening flow, we define the \emph{Entropy} of $\cM$ by
\begin{align}
    \ent[\cM] := \sup_{t\in I} \ent[M_t],
\end{align}
and we say that an ancient flow $\cM$ has \textit{finite-entropy} if $\ent[\cM]<+\infty$.
One important motivation to define $F$-functional and entropy is the well known monotonicity formula  found by Huisken \cite{Hu90}: for any spacetime point $X_0 = (x_0, t_0) \in \R^2\times \R$, for any $t< t_0$
\begin{align}
    \frac{d}{d t} \int_{M_t} \Phi_{X_0}(x,t)\ ds \leq -\int \left\lvert \boldsymbol{\kappa} + \frac{(x - x_0)^\perp}{2(t_0 - t)}\right\rvert^2 \Phi_{X_0}(x,t)\ ds\ .
\end{align}
Here, $\Phi_{X_0}(x, t) = \left( 4\pi (t_0 - t) \right)_+^{-\frac12} \exp\big(-\frac{\lvert x - x_0\rvert^2}{4 (t_0 - t)}\big)$ is the backward heat kernel where $(\cdot)_+ = \max\{ \cdot, 0\}$.
Note that the integral on the left-hand side of the monotonicity formula is exactly $F_{x_0, (t_0 - t)}[M_t]$.
Thus, $t \mapsto \ent[M_t]$ is non-increasing (see \cite[Lemma 1.11]{CM12} for detail.) 

For any immersed $C^2$ planar curve, finite total curvature implies finite entropy (Theorem~\ref{thm: finite entropy}). 
The converse is generally not true. For instance, the graph of $y = \sin(x)$ and the logarithmic spiral have finite entropy but infinite total curvature (see the appendix for the detail).
Nevertheless, we show that when restricting to the class of ancient complete smooth embedded curve shortening flows, these two conditions are equivalent.
\begin{theorem}\label{thm: equivalence}
    Suppose $\cM$ is an ancient complete smooth embedded curve shortening flow. Then $\cM$ has finite total curvature if and only if $\cM$ has finite entropy.
\end{theorem}

Angenent and You \cite{AY21} constructed examples of ancient solutions by gluing together a finite number of grim reapers, called the \emph{ancient trombones}. They showed that their ancient trombones have finite total curvature. In the $2$-dimensional case, Lambert--Lotay--Schulze \cite{LLS} showed that the special Lagrangian surfaces (considered as static ancient solutions) with finitely many planar ends have finite total curvature, by utilizing the correspondence between special Lagrangians and algebraic curves in $\mathbb{C}^{2}$ via hyperk\"ahler rotations. It would be interesting to investigate the relation between the entropy and the total curvature in higher dimensions.

Assuming well-definedness and boundedness of Lagrangian angle, Neves \cite{Neves07} used the Huisken's monotonicity formula to argue that the tangent flows of a Lagrangian mean curvature flow must be a union of Lagrangian cones $C_1, \cdots,  C_{N}$ with a \emph{unique} set of constant angles $\overline{\theta}_{1}, \cdots, \overline{\theta}_{N}$. 
Similar argument was used by Lambert--Lotay--Schulze \cite{LLS} to show the analogous result for tangent flow at infinity for ancient Lagrangian mean curvature flows. 
In curve shortening flow, namely the $1$-dimensional Lagrangian mean curvature flow, the uniqueness of the set of constant Lagrangian angles of the tangent flow directly leads to the uniqueness of tangent flow. Using the simple fact that finite total curvature implies the boundedness of the Lagrangian angle, we give a short proof of \cite[Theorem 1.1]{CSSZ24} in the complete non-compact case:
\begin{corollary}
    The tangent flow at infinity of a smooth, complete, non-compact, embedded ancient solution of the curve shortening flow with finite entropy is unique.
\end{corollary}

This paper is organized as follows: In Section 2, we recall Sturm's theory and its applications to geometric functions: the squared-distance function and the angle function. In addition, we recall the rough convergence theorem~\ref{thm: rough convergence} from \cite{CSSZ24} that gives the sheeting structure for every very-old time-slice of CSF in our setting. 
The combination of Strum's theory and the rough convergence theorem gives a uniform bound of the number of local minimum/maximum points of the squared-distance functions (Proposition~\ref{prop: finite number of minimum points}).
This is the most crucial observation in this paper.
In Section 3.1, we prove that the finite total curvature condition implies the finite entropy condition for curves (Theorem~\ref{thm: finite entropy}).
In Section 3.2, we prove the bound of the number of preimages of (mod $\pi$)-angle function (Proposition~\ref{prop: bounded parallel tangent/normal lines}). 
The proof involves a contradiction argument together with the uniform bound of local maximum points of squared-distance functions (Proposition~\ref{prop: finite number of minimum points}) and a point selection trick (Lemma~\ref{lem: surjectivity}) to find a squared-distance function that violates the bound. 
The control of preimages of (mod $\pi$)-angle function allows us to find the bound of the total curvature for finite entropy CSF in our setting (Theorem~\ref{thm: bounded total curvature}) and a uniform bound of angle function for the non-compact case (Corollary~\ref{cor: bound of angle}). 
In Section 4, we give a short proof of the uniqueness of the tangent flow for the non-compact case (Theorem~\ref{thm:main.unique}) by employing the uniform bound of the angle function.

\section*{Acknowledgement}
The authors would like to thank the National Center for Theoretical Sciences (NCTS) for providing the excellent research environment and support. The authors thank Kyeongsu Choi and Donghwi Seo for inspiring discussions, which directly leads to the results in this paper.

\section{Preliminaries}

\subsection{Sturm's Theory}\label{sec: 2.1}
For a function $u: [A_1, A_2] \to \R$, we will call the set $u^{-1}(0)$ of zeros the \emph{nodal set} of $u$. Any connected component of $[A_1, A_2] \setminus u^{-1}(0)$ is called a \emph{nodal domain}. A point $x_0\in (A_0, A_1)$ is called a \emph{multiple zero} of $u$ if $u(x_0) = u_x (x_0) = 0$.
\begin{proposition}[Sturm's theory \cite{Ang:1988:ZSS}]\label{prop: zeroset}
Let $u: [A_1, A_2]\times [0, T] \to \mathbb{R} $ be a smooth nontrivial solution to the parabolic equation $u_t=au_{xx}+bu_x + cu$, where $a,b,c\in C^{\infty}([A_1, A_2]\times [0,T])$ and $a>0$. Suppose $u$ satisfies Dirichlet, Neumann, the periodic boundary condition, or the non-vanishing boundary condition: $u(A_i,t) \neq 0$ holds for all $t\in [0,T]$ and $i = 1,2$. Let $z(t)$ denote the number of zeros of $u(\cdot, t)$ in $[A_0, A_1]$ with $t\in [0, T]$. Then 
\begin{enumerate}
    \item[$(a)$] for all $t\in [0, T]$, $z(t)$ is finite,
    \item[$(b)$] if $(x_0, t_0)$ is a multiple zero of $u$, i.e., $u$ and $u_x$ both vanish at $(x_0, t_0)$, then for all $t_1< t_0 < t_2$ we have $z(t_1) > z(t_2)$. More precisely, there exists a neighborhood $N = [x_0 - \vare , x_0 + \vare] \times [t_0 - \delta, t_0 + \delta]$ such that 
    \begin{enumerate}
        \item[$(i)$] $u(x_0 \pm \vare, t) \neq 0$ for $\lvert t - t_0 \rvert \leq \delta$,
        \item[$(ii)$] $u(\cdot, t + \delta)$ has at most one zero in the interval $[x_0 - \vare , x_0 + \vare]$,
        \item[$(iii)$] $u(\cdot, t - \delta)$ has at least two zeros in the interval $[x_0 - \vare , x_0 + \vare]$.
    \end{enumerate}   
\end{enumerate} 
\end{proposition}

\begin{corollary}\label{cor: almost no multiple zero}
Let $u$ be assumed as in the above proposition. Then the set 
\begin{equation*}
    \{t\in [0, T]: u(\cdot, t) \mbox{ has a multiple zero}\}
\end{equation*} 
is finite.
\end{corollary}

\subsection{Ancient paths of knuckles and tips}\label{subsec: knuckle/tip}
In this subsection, we assume that $\cM = \{M_t\}_{t\in I}$ is an ancient complete embedded smooth solution to the CSF with finite entropy parametrized by $\ga$. 
Given $x_0\in \R^2$, define the \emph{squared-distance function} with respect to $x_0$ by $\varphi^{(x_0)}(s,t) = |\ga(s,t) - x_0|^2 + 2t$. 
A direct computation gives
\begin{align}\label{eq: derivatives of distance function}
    \tfrac{\partial}{\partial s}\varphi^{(x_0)} = 2 \langle \ga(s, t) - x_0, \bt(s, t) \rangle, \quad \tfrac{\partial^2}{\partial s^2} \varphi^{(x_0)} = 2 + 2\langle \ga(s,t) - x_0, \boldsymbol{\kappa}(s,t)\rangle.
\end{align}
Thus, $\varphi^{(x_0)}$ satisfies the heat equation:
\begin{align}\label{eq: heat.dist}
    \tfrac{\partial}{\partial t}\varphi^{(x_0)}=\tfrac{\partial^2}{\partial s^2} \varphi^{(x_0)}.
\end{align}

\begin{lemma}[\cite{CSSZ24} Lemma~2.6]\label{lem: center of distance}
   $\varphi^{(x_0)}(\cdot, t_0)$ attains a critical point at $s_0$ if and only if $x_0$ is on the normal line of $\ga$ at $\ga(s_0,t_0)$, that is, for some $\al \in \R$
\begin{equation}\label{eq: center}
    x_0 = \ga(s_0,t_0) - \al \bn(s_0,t_0).
\end{equation}
Furthermore, if $\ka(s_0,t_0) = 0$, then $\varphi^{(x_0)}(\cdot, t_0)$ always attains a local minimum at $s_0$; if $\ka(s_0,t_0) \neq 0$ and we substitute $\alpha = \beta / \ka(s_0, t_0)$ in $(\ref{eq: center})$, then
\begin{enumerate}
    \item[$(a)$] when $\beta > 1$, $\varphi^{(x_0)}(\cdot, t_0)$ attains a local maximum at $s_0$, 
    \item[$(b)$] when $\beta < 1$, $\varphi^{(x_0)}(\cdot, t_0)$ attains a local minimum at $s_0$,
    \item[$(c)$] when $\beta = 1$, $x_0$ is the center of the osculating circle tangent to $\ga$ at $\ga(s_0,t_0)$, that is,  $\varphi^{(x_0)}_s(s_0,t_0) = \varphi^{(x_0)}_{ss}(s_0,t_0) = 0$.
\end{enumerate}
\end{lemma}

\begin{proposition}[ancient paths of knuckles/tips, \cite{CSSZ24} Proposition~2.7]\label{prop: ancient local-min-path/local-max-path}
Suppose that for some $t_0\in I$ such that $p_0$ and $q_0$ are, respectively, a local minimum point, a local maximum point of $\varphi^{(x_0)}(\cdot, t_0)$. Then the following statements hold.
\begin{enumerate}
    \item[$(a)$] There exist continuous ancient paths $\bp = \bigcup_{t\leq t_0} (p(t), t), \bq = \bigcup_{t\leq t_0} (q(t), t)\subset \mathbb{L}\times (-\infty, t_0]$ such that $p(t_0) = p_0$, $q(t_0) = q_0$ and $p(t), q(t)$ are a local minimum point and a local maximum point of $\varphi^{(x_0)}(\cdot, t)$, respectively, for all $t\leq t_0$. Furthermore, for all $t_1 \leq t_2 \leq t_0$,
    \begin{equation}\label{eq: monotone of local extremum}
        \varphi^{(x_0)}(\bp(t_1)) \leq \varphi^{(x_0)}(\bp(t_2)) \quad \text{and} \quad  \varphi^{(x_0)}(\bq(t_1)) \geq  \varphi^{(x_0)}(\bq(t_2)).
    \end{equation}

    \item[$(b)$] The sign of curvature $\ka(\bq(t))$ remains the same for all $t\leq t_0$.

    \item[$(c)$] $(-t)^{-1}|\ga(\bp(t))|^2$ is increasing in $t$.
    
    \item[$(d)$] 
        \begin{equation}\label{eq: limsup of minimum}
            \limsup_{t\to -\infty} \,(-t)^{-1}|\ga(\bp(t))|^2 \leq  2
        \end{equation}
        and
        \begin{equation}\label{eq: liminf of maximum}
            \liminf_{t\to -\infty} \,(-t)^{-1}|\ga(\bq(t))|^2 \geq  2.
        \end{equation}
\end{enumerate} 
\end{proposition}

\subsection{The angle function}\label{sec: 2.3}
Let $\cM = \{M _t\}_{t\in I}$ be an ancient complete \emph{non-compact} embedded smooth solution to  CSF parametrized by $\ga: \R \times I \to \R^2$.  
We can choose a smooth global function $\th: \R\times I\to \R$ which is an extension of the angle measured from $e_1 = (1, 0)$ to $\bt(p,t)$ counterclockwisely mod $2\pi$. 
It is well-known that $\th_s = \kappa$, where the sign is positive in our convention of curvature.
Equivalently, this is the well-known identity for the Lagrangian angle
\begin{equation*}
    \boldsymbol{\kappa} = J \D \th.
\end{equation*}

\begin{lemma}\label{lem: angle equation}
Suppose that $\cM$ is a complete non-compact embedded smooth solution to CSF. The angle function satisfies the heat equation
\begin{equation}\label{eq: heat.angle}
    \tfrac{\pa}{\pa t}\th = \tfrac{d^2}{d s^2} \th.
\end{equation}
Moreover, 
\begin{equation}\label{eq: heat.angle2}
    \tfrac{\pa}{\pa t} \th^2 = \tfrac{d^2}{d s^2}\theta^2  - 2 \th_s^2.
\end{equation}
\end{lemma}
\begin{proof}
Note that \cite[Remark 2.3.2 page 34]{Man11} the commutation rule $\pa_t\pa_s = \pa_s \pa_t + \kappa^2 \pa_s$ on CSF implies that
\begin{align}\label{eq: evol.bt}
    \pa_t \bt = \kappa_s \bn.
\end{align}
By definition of $\theta$, we can write $\bt = (\cos \theta, \sin \theta)$. Plugging this into \eqref{eq: evol.bt} and comparing the coefficients gives $\theta_t = \kappa_s = \theta_{ss}$. Then \eqref{eq: heat.angle2} follows from chain rule for heat operator \cite[Lemma 3.14]{Eck2004RTM}.
\end{proof}

\begin{proposition}\label{prop: finite range of angle}
Suppose that $\cM$ is an ancient complete non-compact embedded smooth solution to CSF. Then the range of the angle of $M_t$
\begin{equation}
    \mathcal{A}(t):= \{\th(s, t): s\in \R\} \subseteq \R
\end{equation}
is non-increasing in the sense that $\mathcal{A}(t_1) \supseteq \mathcal{A}(t_2)$ for $t_1 < t_2$.
\end{proposition}
\begin{proof}
We may assume that $\cM$ is not a quasi-static line; otherwise, the claim holds trivially. Proposition~\ref{prop: zeroset} applied to the solution $\theta - c$ of the same heat equation in Lemma~\ref{lem: angle equation} implies that for any $c\in \mathcal{A}(t_2)$, the zero of $\th(\cdot, t) - c$ exists for all $t< t_2$. Then the claim follows.
\end{proof}

\subsection{Structure theorem for curve shortening flow with finite entropy}
In the rest of the paper, we assume that the ancient solution $\cM$ is defined on $(-\infty, 0]$. 
Let $\overline{M}_\tau = e^{\frac{\tau}{2}} M_{-e^{-\tau}}$ denote the rescaled flow. We have the following sheeting result for ``very-old" time-slices $\overline{M}_{\tau}$:

\begin{theorem}[Rough convergence, \cite{CSSZ24} Theorem~3.3]\label{thm: rough convergence}
Let $\cM$ be a smooth embedded ancient flow with entropy $\Lambda:= \ent[\cM]<\infty$ which is not a family of shrinking circles. Then $\Lambda = m\in\mathbb{N}$, and the following property holds for its rescaled flow $\overline{M}_\tau$$:$

Given $R \gg 1$ and $\varepsilon>0$, there are $T = T(\varepsilon, R)\ll -1$ and a smooth rotation function $S:(-\infty,T]\to SO(2)$ such that $S(\tau) \overline{M}_\tau \cap B_R(0)$ is a disjoint union of $m$ graphs, each of which is $\varepsilon$-close in $C^{\lfloor\frac{1}{\varepsilon}\rfloor}$-sense to the coordinate axis $\ell_{0} = \{(y,0): |y|<R\}$.
\end{theorem}

A direct consequence of Theorem~\ref{thm: rough convergence} and Proposition~\ref{prop: ancient local-min-path/local-max-path} is the following:
\begin{proposition}[maximal number of local min/max points]\label{prop: finite number of minimum points} 
Under the assumptions of \ref{thm: rough convergence}, for any $x_0 \in \R^2$, $t\in I$, $\varphi^{(x_0)}(\cdot, t)$ has at most $m$ local minimum points. Furthermore, if $\mathbb{L} = \R$ $($resp. $\mathbb{L} = \mathbb{S}^1$$)$, then $\varphi^{(x_0)}(\cdot, t)$ has at most $m-1$ $($resp. $m$$)$ local maximum points.
\end{proposition}

\section{Proof of main theorem}
\subsection{Finite total curvature implies finite entropy}

\begin{lemma}[entropy of graph with finite slope]\label{lem: entropy of graph}
    Let $f: (a, b)\rightarrow \R$ be a differentiable function satisfying $\lvert f'(x) \rvert \leq \eta$ for all $x\in (a, b)$ for some constant $\eta$. Then the entropy of the graph $y = f(x)$ is less than or equal to $\sqrt{1 + \eta^2}$.
\end{lemma}

\begin{proof}
    We can always extend $f$ to $\R$ with $\lvert f'(x)\rvert \leq \eta$ everywhere. Denote by $\Gamma_f$ the graph of $y = f(x)$.
    Using the derivative bound, for all $\lambda > 0$
    \begin{align}
        F_{0,\lambda^2 }[\Gamma_f] = \tfrac{1}{\sqrt{2\pi}} \int_{-\infty}^\infty \exp\big(-\frac{x^2 + \lambda^{-2} f(\lambda x)^2}{4}\big)\, \sqrt{1 + \lvert f'(\lambda x) \rvert^2}\, dx \leq \sqrt{1 + \eta^2}.
    \end{align}
    In the last inequality, we simply throw away non-positive $-\lambda^{-2} f^2$ term in the exponent, so 
    the inequality is invariant under translation.
    Therefore, $\ent[\Gamma] \leq \sqrt{1+ \eta^2}$.
\end{proof}

\begin{lemma}\label{lem: ent bound}
    Suppose that $\Gamma$ is an immersed, connected, $C^2$ planar curve with total curvature $\int_\Gamma |\kappa|\, ds \leq \tfrac{\pi}{4}$.
    Then $\ent[\Gamma] \leq \sqrt{2}$.
\end{lemma}

\begin{proof}
    Let $\gamma: (a,b) \rightarrow \R^2$ be an arc-length parametrization for $\Gamma$ where $(a, b)\ni 0$. 
    Since the entropy is invariant under rotation on $\R^2$, applying a suitable rotation we may assume that $\gamma(0) = P$ and $\theta(0) = 0$ for some interior point $P\in\Gamma$.
    From the assumption on total curvature, we readily see that $\lvert \theta(s)\rvert \leq \pi/4$. 
    Thus, $\Gamma$ can be expressed as a (part of) graph of an entire function $f(x)$ over the $x$-axis satisfying $\lvert f'(x)\rvert \leq 1$ everywhere. 
    By Lemma~\ref{lem: entropy of graph}, $\ent[\Gamma] \leq \sqrt{2}$.
\end{proof}

\begin{theorem}\label{thm: finite entropy}
    Suppose that $\Gamma$ is an immersed, connected, $C^2$ planar curve with total curvature $\int_\Gamma |\kappa|\, ds \leq \alpha \pi$ for some integer $\alpha < \infty$.
    Then $\ent[\Gamma] \leq 4\alpha \sqrt{2}$.
\end{theorem}

\begin{proof}
    By the total curvature condition, we can find a partition of $\Gamma$ into $4\alpha$ connected components $\Gamma_1, \ldots, \Gamma_{4\alpha}$ such that $\int_{\Gamma_i} \lvert \kappa\rvert\, ds \leq \pi/4$ for all $i = 1, \ldots, 4\alpha$.
    Applying Lemma~\ref{lem: ent bound} to each component, for any $\lambda > 0$, $x_0\in \R^2$
    \begin{align}
        F_{x_0, \lambda}[\Gamma] \leq \sum_{i = 1}^{4\alpha} F_{x_0, \lambda}[\Gamma_i] \leq 4\alpha \sqrt{2}. 
    \end{align}
\end{proof}

\subsection{Finite entropy implies finite total curvature}
Despite the fact that the converse of Theorem~\ref{thm: finite entropy} is not true for arbitrary curves, in this subsection we show that the finite total curvature and the finite entropy conditions are equivalent for ancient complete embedded smooth CSF.

\begin{lemma}[surjectivity for intersection of normal lines near infinity]\label{lem: surjectivity}
    Let $0 < \delta <1$ and let $\gamma(s) : ( -2\delta, 2\delta ) \rightarrow \R^2$ be an arc-length parametrized $C^2$ curve so that $\kappa(0) = \kappa_0 > 0$ and 
    \begin{align}\label{eq: curv.bound}
        \tfrac12 \kappa_0 \leq \kappa(s) \leq 2\kappa_0 \quad \mbox{for all } \; s\in (-\delta , \delta).
    \end{align}
    Suppose $\bt(0) = \gamma'(0) = (1,0)$. Then there exists $R >0 $ depending only on $\kappa_0$ and $\gamma(0)$ such that for all $Y \geq R$, there is $s\in ( - \delta, \delta)$ such that the normal line at $\gamma(s)$ intersects $y$-axis at $(0, Y)$.
\end{lemma}

\begin{proof}
    Let $\gamma(s) = (x(s), y(s)) \in \R^2$ and let $\theta(s)$ denote the smooth angle function of the unit tangent vector $\bt(s)$ measured from the positive $x$-direction counterclockwisely so that $\theta(0) = 0$. We can write $\bt(s) = (\cos\theta(s), \sin \theta(s))$. The normal line at $\gamma(s)$ passes through $(0, Y)$ if and only if
    \begin{align}\label{eq: intersection}
        0 = \big( (0, Y) - \gamma(s) \big) \cdot \bt(s) = -x(s) \cos \theta(s) + (Y - y(s))\sin \theta(s).
    \end{align}

    Let $\gamma(0) = (x_0, y_0)$. We may assume that $x_0 \neq 0$; otherwise, \eqref{eq: intersection} holds trivially for all $Y$ with $s = 0$. We may WLOG assume that $x_0 >0$. Then, the solution of \eqref{eq: intersection} is given by
    \begin{align}\label{eq: intercept}
        Y = y(s) + x(s)\cot \theta(s).
    \end{align}
    We only focus on \emph{positive} intercept $Y$. From \eqref{eq: curv.bound} and $\theta(0) = 0$, for $0 \in (0, \delta]$, 
    \begin{align}\label{eq: angle.bound}
        \tfrac12 \kappa_0 s \leq \theta(s) = \int_0^s \kappa(s)\, ds \leq 2 \kappa_0 s.
    \end{align}
    We can make $\delta$ smaller such that $\kappa_0 \delta < \tfrac{\pi}{4}$. Using $\gamma(s) = \gamma(0) + \int_0^s \bt(s)\, ds$ and $\delta < 1$, for all $s \in (0, \delta]$,
    $x_0 \leq x(s) \leq x_0 + 1$ and $y_0 \leq y(s) \leq y_0 + 1$.
    In particular, when $s = \delta$, 
    \begin{align}\label{eq: R}
        Y\leq (y_0+ 1) + (x_0 + 1) \cot(\tfrac12 \kappa_0\delta) =: R.
    \end{align}
    Furthermore, by \eqref{eq: intercept} and \eqref{eq: angle.bound} as $s \to 0^+$ we find
    \begin{align}\label{eq: limit of Y}
        Y \geq y_0 + x_0 \cot (2\kappa_0 s) \to +\infty.
    \end{align}
    Combining \eqref{eq: R} and \eqref{eq: limit of Y} together with continuity, for all $Y \geq R$, there exists $s\in (0, \delta]$ that solves \eqref{eq: intercept}. Finally, if $x_0 < 0$, we can apply an analogue argument for $s\in [-\delta, 0)$.
\end{proof}

\begin{lemma}\label{lem: unbdd dist}
    Let $\Gamma\subset \R^2$ be a complete non-compact embedded curve with finite entropy, and let $\gamma : \R\rightarrow \Gamma$ be an arc-length parametrized embedding. 
    Then for any $x_0 \in \R^2$, $\lvert \gamma(s) - x_0 \rvert \to \infty$ as $s \to \pm \infty$.
\end{lemma}

\begin{proof}
    By a proper translation of $\Gamma$, we can always assume that $x_0$ is the origin.
    Suppose that the claim is not true as $s \to \infty$; that is, there exist $R>0$ and a sequence $s_i \to \infty$ such that $\lvert \gamma(s_i) \rvert \leq R$.
    By compactness of closed ball $\bar{B}_R$, after passing to a subsequence and relabeling, we may further assume that $\lvert \gamma(s_j) - \bp \rvert \leq 1$ for some $\bp\in \bar{B}_R$ and  $s_{i+1} - s_i\geq 2$ for all $i = 1, 2, \ldots$. 
    Note that the curved segments $\gamma\big([s_i - 1, s_i +1]\big)$ are mutually disjoint due to embeddedness and all have intrinsic length equal to $2$.
    Furthermore, for any $s\in [s_i - 1, s_i +1]$, by triangle inequality,
    \begin{align}
        \lvert \gamma(s) - \bp \rvert \leq \lvert \gamma(s) - \gamma(s_i)\rvert + \lvert\gamma(s_i) - \bp\rvert \leq 1 + 1 = 2.
    \end{align}
    Therefore,
    \begin{align*}
        \ent[\Gamma] \geq F_{\bp, 1}[\Gamma] \geq \sum_{i = 1}^\infty \int_{s_i - 1}^{s_i +1} \tfrac{1}{\sqrt{2\pi}} e^{-\frac{\lvert \gamma(s) - \bp \rvert^2 }{4}} \, ds \geq \sum_{i = 1}^\infty \tfrac{e^{-1}}{\sqrt{2\pi}}  \int_{s_i - 1}^{s_i +1}\, ds = \sum_{i = 1}^\infty \tfrac{e^{-1}}{\sqrt{2\pi}}\cdot 2 = \infty,
    \end{align*}
    which contradicts the finite entropy condition. For the case $s \to-\infty$, we just change the orientation of $\gamma$.
\end{proof}

\begin{proposition}[bounded number of parallel normal lines]\label{prop: bounded parallel tangent/normal lines}
Suppose that $\cM = \{M_{t}\}_{t\in I}$ is an ancient complete non-compact $($resp. closed $)$ embedded smooth solution to CSF with $\ent(\cM) = m < \infty$. Then
\begin{enumerate}
    \item[$(a)$] For all $t\in I$, there are at most $m-1$ $($resp. $m$ $)$ distinct points on $M_t$ having nonzero parallel curvature vectors with the same sign.
    \item[$(b)$] For all $t\in I$, there are at most $2m - 1$ $($resp. $2m + 1$ $)$ distinct points on $M_t$ having parallel normal lines.
    \item[$(c)$] For all $t\in I$, for all $\xi \in \R$, $\th(\cdot, t) \equiv \xi \mod \pi$ has at most $2m - 1$ $($resp. $2m + 1$ $)$ solutions.
\end{enumerate}
\end{proposition}

\begin{proof}
    We first consider the non-compact case. Suppose (a) were false. Up to a proper rotation and a change of orientation, we may assume that there exist $q_1 < q_2 <\ldots < q_m$ for some $t\in I$ at which the unit normal vector $\bn(q_i, t) = ( 1, 0)$ and curvature $\kappa( q_i, t) > 0$ for all $i = 1, \ldots, m$. 
    Applying Lemma~\ref{lem: surjectivity} to all $q_1, \ldots, q_m$, there exists $R>0$ depending on the position $\gamma$ and curvature around these points so that for any $Y \geq R$, for each $q_i$ we can find $q_i'$ near $q_i$, at which the normal line to $M_t$ passing through the common intersection point $(0, Y)$. 
    According to Lemma~\ref{lem: center of distance}, by choosing $Y$ even larger (depending on the curvature at all $q_i'$), the squared-distance function $\varphi^{(x_0)}$ with respect to the center $x_0 := (0, Y)$ defined on $M_t$ achieves strict local maximums at all $q_i'$.
    By Lemma~\ref{lem: unbdd dist}, $\varphi^{(x_0)}(s)$ diverges to $+\infty$ as $s\to \pm \infty$. There must be $m$ distinct local minimum points $p_1, \ldots, p_m$ of $\varphi^{(x_0)}(s)$ such that $p_i < q_i < p_{i+1}$ for $i = 1, \ldots, m-1$. This contradicts Proposition~\ref{prop: finite number of minimum points}.

    Note (b) and (c) are equivalent, so it suffices to show (c). Suppose there are distinct $s_1, \ldots, s_{2m}$ zeros of $\theta(s_i, t) = \xi \mod \pi$ for some $\xi$ and $t\in I$. Recall that $\theta(s,t)$ satisfies the heat equation in Lemma~\ref{lem: angle equation}. 
    By Proposition~\ref{prop: zeroset} and Corollary~\ref{cor: almost no multiple zero}, by perturbing $t$ backward slightly, we still have $2m$ distinct zeros $s_1', \ldots, s_{2m}'$ of the equation $\theta(\cdot, t) = \xi \mod \pi$ and $\kappa(s_i', t) = \theta_s(s_i', t) \neq 0$ for all $i = 1, \ldots, 2m$. At least $m$ of $s_i$'s have the same sign of curvature. This leads to a scenario that contradicts (a).

    For the closed case, the only difference is the maximal number of local maximum points in Proposition~\ref{prop: finite number of minimum points} and the proof proceeds similarly as the non-compact case.
\end{proof}

\begin{theorem}[bounded total curvature]\label{thm: bounded total curvature}
Suppose that $\cM = \{M_{t}\}_{t\in I}$ is an ancient complete non-compact $($resp. closed $)$ embedded smooth solution to CSF with $\ent(\cM) = m < \infty$. Then
\begin{equation}
    \sup_{t\in I}\int_{M_t} |\ka|\, ds \leq (2m - 1)\pi \quad  (\mbox{resp.}\; (2m + 1)\pi ).
\end{equation}
\end{theorem}

\begin{proof}
We first prove the result in the non-compact case. For any fixed $t\in I$, applying area formula to the angle function $\th(\cdot, t): \R \to \R$, we have
\begin{equation*}\label{eq: area formula}
    \int_\R |\th_s(s, t)| \, ds = \int_\R \cH^0 \big( \th(\cdot, t)^{-1}(\xi) \big) d\xi,
\end{equation*}
where $\cH^0$ is the counting measure. 
Using $\th_s(\cdot, t) = \ka(\cdot, t)$ and Tonelli's theorem, we get
\begin{align*}
    \int_{M_t}|\ka|\, ds &= \int_\R \cH^0 \big( \th(\cdot, t)^{-1}(\xi) \big)\, d\xi
    = \sum_{k = -\infty}^\infty \int_{k\pi}^{(k+1)\pi} \cH^0\big(\th(\cdot, t)^{-1}(\xi)\big) d\xi\\
    &= \int_{0}^{\pi} \sum_{k = -\infty}^\infty  \cH^0\big(\th(\cdot, t)^{-1}(\xi + k\pi)\big) d\xi
    \leq (2m - 1) \pi.
\end{align*}
In the last inequality, we use Proposition \ref{prop: bounded parallel tangent/normal lines} (c).

Note that the above proof does not really need a globally defined smooth angle function, since we only used the (mod $ \pi$)-angle in Proposition~\ref{prop: bounded parallel tangent/normal lines}. In closed case, we can find a (discontinuous) angle function $\theta : [0, \ell) \rightarrow \R$ on $M_t$ parametrized by arclength having a jump of value $2\pi$ at $s= 0$, where $\ell$ is the length of $M_t$. The proof proceeds similarly as the non-compact case.
\end{proof}

\begin{proof}[Proof of Theorem~\ref{thm: equivalence}]
    Combining Theorem~\ref{thm: finite entropy} and Theorem~\ref{thm: bounded total curvature} completes the proof.
\end{proof}

\begin{corollary}[uniform bound for the angle]\label{cor: bound of angle}
    Suppose that $\cM = \{M_{t}\}_{t\in I}$ is an ancient complete \emph{non-compact} embedded smooth solution to CSF with $\ent(\cM) = m < \infty$ and suppose that $\theta(0, t_0)\in [0, 2\pi)$ for some $t_0\in I$. 
    Then $\vert\theta(s, t)\rvert \leq (2m + 1)\pi$ for all $s\in \R, t\in I$.
\end{corollary}

\begin{proof}
    Let $\theta_0 = \theta(0, t_0)$. By Proposition~\ref{prop: finite range of angle}, it suffices to show that $\theta$ is bounded for $t < t_0$ and for any $t < t_0$, there exists $s_t$ such that $\theta(s_t, t) = \theta_0$. Then Theorem~\ref{thm: bounded total curvature} implies that for any $s$,
    \begin{align*}
        \lvert \theta(s,t) \rvert = \left\lvert \theta_0 + \int_{s_t}^s \kappa\, ds \right\rvert \leq 2\pi + (2m - 1) \pi = (2m + 1)\pi.
    \end{align*}

\end{proof}

\section{Application: uniqueness of tangent flow at infinity for non-compact curves}

We prove a uniqueness result for tangent flow at infinity, making use of the well-definedness of the angle function $\theta$ in the complete non-compact case. Note that by Corollary~\ref{cor: bound of angle}, $\theta$ is uniformly bounded in terms of the entropy, so the result is exactly the same as Neves' result \cite{Nev07} and Lambert--Lotay--Schulze's \cite{LLS} in the $1$-dimensional case. We remark that our uniqueness result below does not provide a convergence rate, whereas the uniqueness result in \cite[Theorem~1.2]{CSSZ24} gives an exponential convergence of the rescaled flow $\overline{M}_{\tau}$ as $\tau\to-\infty$.

\begin{theorem}[unique tangent flow at infinity]\label{thm:main.unique}
An ancient finite-entropy smooth complete non-compact curve shortening flow embedded in $\mathbb{R}^2$ has a unique tangent flow at infinity. Moreover, the tangent flow at infinity is a straight line with multiplicity $m\geq 2$, unless the flow is a static line.
\end{theorem}

\begin{proof}
    To show the uniqueness of tangent flow at infinity, from rough convergence theorem, it suffices to show that the rotation function $S(\tau)$ converges as $\tau \to -\infty$.
    Let $\theta$ be a choice of angle function defined in Subsection~\ref{sec: 2.3} such that $\theta(0, t_0) \in [0, 2\pi)$ for some $t_0 \in I$.
    Suppose that $\ent[\cM] = m$.
    By Corollary~\ref{cor: bound of angle}, the angle function $\lvert \theta(s, t)\rvert$ is uniformly bounded by $(2m + 1)\pi$ for all $t\in I$ and $s\in \R$.
    Let $\vare > 0$.
    By rough convergence theorem~\ref{thm: rough convergence}, there exist $R \gg 0$, $T \ll -1$, and some fixed $a_1, \ldots, a_m\in \mathbb{Z}$ such that for all $\tau \leq T$, for all $i$,
    \begin{align}\label{eq: angle of sheet}
        \left\lvert \tfrac{1}{\sqrt{2\pi}} \int_{\Sigma_\tau^i \cap B_R} \theta^2 e^{-\frac{\lvert x \rvert^2}{4}} \, ds - \big(S(\tau) + a_i \pi\big)^2 \right\rvert < \vare.
    \end{align}
    where $\Sigma_\tau^1, \ldots, \Sigma_\tau^m$ are the connected components of $\overline{M}_\tau \cap B_R$.
    Here, we identify the smooth rotation function $S(\tau)$ with a real-valued angle function.
    From \cite[Lemma 2.4]{CIM15}, $F_{0, 1}[\overline{M}_\tau \setminus B_R] \leq Cm e^{-\frac{R^2}{8}}$ for some numerical constant $C$. 
    As a result, 
    \begin{align}\label{eq: angle outside}
        \left\lvert \tfrac{1}{\sqrt{2\pi}} \int_{\overline{M}_\tau \setminus B_R} \theta^2 e^{-\frac{\lvert x \rvert^2}{4}} \, ds \right\rvert \leq C m (2m +1)^2 \pi^2  e^{-\frac{R^2}{8}},
    \end{align}
    which can be made arbitrarily small by choosing a sufficiently large $R$. 
    Combining \eqref{eq: angle of sheet} and \eqref{eq: angle outside} gives
    \begin{align}\label{eq: limit identity}
        \lim_{\tau \to -\infty} \left\lvert \tfrac{1}{\sqrt{2\pi}} \int_{\overline{M}_\tau} \theta^2 e^{-\frac{\lvert x \rvert^2}{4}} \, ds - \sum_{i = 1}^{m} \big(S(\tau) + a_i \pi\big)^2 \right\rvert = 0.
    \end{align}
    By weighted monotonicity formula \cite[Theorem 4.13]{Eck2004RTM} and Lemma~\ref{lem: angle equation} \eqref{eq: heat.angle2},
    \begin{align*}
        \frac{d}{d t}\int_{M_t} \theta^2 \Phi_{(0, 0)} \, ds &= \int_{M_t} \Big[(\tfrac{\pa}{\pa t} - \tfrac{\pa^2}{\pa s^2})\theta^2 - \Big\lvert \kappa + \frac{x^\perp}{2t}\Big\rvert^2\theta^2\Big] \Phi_{(0, 0)}\, dt\\
        &= -\int_{M_t} \Big[2\kappa^2 + \Big\lvert \kappa + \frac{x^\perp}{2t}\Big\rvert^2\theta^2\Big] \Phi_{(0, 0)}\, dt \leq 0.
    \end{align*}
    It follows from the above monotonicity, uniform bounds of angle and entropy, and \eqref{eq: limit identity} that the limit
    \begin{align*}
        \lim_{t\to -\infty}\int_{M_t} \theta^2 \Phi_{(0, 0)} \, ds = \lim_{\tau\to -\infty} \tfrac{1}{\sqrt{2\pi}} \int_{\overline{M}_\tau} \theta^2 e^{-\frac{\lvert x \rvert^2}{4}} \, ds = \lim_{\tau \to -\infty} \sum_{i = 1}^m (S(\tau) + a_i \pi)^2
    \end{align*}
    exists. This implies that $\lim_{\tau \to -\infty} S(\tau)$ exists.
\end{proof}

\appendix
\section{The entropy bound for the logarithmic spiral}

The converse of Theorem~\ref{thm: finite entropy} is generally not true for $C^2$ embedded curves in $\R^2$. For example, by Lemma~\ref{lem: entropy of graph} the graph of $y = \sin(x)$ has finite entropy but infinite total curvature. Another interesting example is the logarithmic spiral \cite[Logarithmic\_spiral]{Wiki}, which is self-similar: a scaled logarithmic spiral is congruent to the original curve by a rotation. 
Given real constants $a > 0$ and $k\neq 0$, the logarithmic spiral  can be written as 
    \begin{align}
        S_{a, k } := \{(r\cos\varphi, r\sin \varphi)\in \R^2\, :\, r = a e^{k\varphi}, \varphi\in \R\}. 
    \end{align}
    Since the pitch angle of $S_{a,k}$ is a constant $\alpha = \arctan k$, the angle function $\theta$ defined in Section~\ref{sec: 2.3} at point $(r\cos\varphi, r\sin \varphi)$ is given by $\varphi + \pi/2 - \alpha$. 
    Then unbounded $\theta$ implies that the total curvature of $S_{a,k}$ is infinite.
    For any $0\leq r_1 < r_2$, the length of $S_{a, k} \cap (B_{r_2}(0) \setminus B_{r_1}(0))$ is given by 
    \begin{align}\label{eq: arclength.log}
        L[S_{a, k} \cap (B_{r_2}(0) \setminus B_{r_1}(0))] = \frac{r_2 - r_1}{\lvert \sin \alpha \rvert}.
    \end{align}
    Note that $S_{a, k}\cap B_1(0)$ has finite length, which implies that the logarithmic spiral $S_{a, k}$ is not a complete curve. 
    Finally, we will show that the entropy of $S_{a,k}$ is finite.

\begin{proposition}[entropy bound of logarithmic spiral]
    Let $a >0 ,k \neq 0$, $\alpha = \arctan k$, and let $S_{a,k}$ be defined as above.
    Then $\ent[S_{a,k}] \leq  1/\lvert \sin \alpha \rvert < \infty $.
\end{proposition}

\begin{proof}
    Let $x_0 \in \R^2$ and $\lambda > 0$. 
    By \eqref{eq: arclength.log}, the arclength element is $ds = \frac{dr}{\lvert \sin \alpha \rvert}$ in 
    the paramaterization by $r$.
    Using triangle inequality and the substitution $u = (r - \lvert x_0 \rvert)/\lambda$,
    \begin{align*}
        F_{x_0, \lambda^2}[S_{a,k}] &= 
        \frac{1}{\sqrt{2\lambda^2\pi}} \int_{0}^\infty \exp\big(-\tfrac{\lvert x - x_0\rvert^2}{4\lambda^2}\big) \, \frac{d r}{\lvert \sin\alpha\rvert} \\
        &\leq \frac{1}{\lvert \sin \alpha \rvert \sqrt{2\lambda^2 \pi}} \int_{0}^{\infty} \exp\big(-\tfrac{( r - \lvert x_0\rvert )^2}{4\lambda^2}\big)\, d r \\
        &= \frac{1}{\lvert \sin \alpha \rvert \sqrt{2\pi}} \int_{-\lvert x_0 \rvert / \lambda}^{\infty} \exp\big( -\tfrac{u^2}{4}\big)\, du\\
        &\leq \frac{1}{ \lvert \sin \alpha \rvert}.
    \end{align*}
    Thus, $\ent[S_{a,k}] \leq  1/\lvert \sin \alpha \rvert$.
\end{proof}

\bibliographystyle{amsplain} 
\bibliography{CSF}

\end{document}